\documentclass[
final
 , nomarks
]{dmtcs-episciences}

\usepackage[utf8]{inputenc}
\usepackage{subfigure}

\usepackage{amsmath,amsfonts,amssymb,amsthm,
mathtools,xcolor,tikz}

\newtheorem{theorem}{Theorem}

\newtheorem{lemma}[theorem]{Lemma}
\newtheorem{corollary}[theorem]{Corollary}
\newtheorem{conjecture}[theorem]{Conjecture}
{
\theoremstyle{definition}
\newtheorem{definition}[theorem]{Definition}
\newtheorem{example}[theorem]{Example}

}
{
\theoremstyle{remark}

}
\DeclareMathOperator{\st}{st}
\DeclareMathOperator{\maj}{maj}
\DeclareMathOperator{\Av}{Av}
\DeclareMathOperator{\Des}{Des}
\DeclareMathOperator{\Destop}{Destop}
\DeclareMathOperator{\Desbot}{Desbot}
\DeclareMathOperator{\Asc}{Asc}
\DeclareMathOperator{\Asctop}{Asctop}
\DeclareMathOperator{\Ascbot}{Ascbot}
\DeclareMathOperator{\Peak}{Peak}
\DeclareMathOperator{\Val}{Val}

\DeclareMathOperator{\LRmax}{LRmax}
\DeclareMathOperator{\LRmin}{LRmin}
\DeclareMathOperator{\RLmax}{RLmax}
\DeclareMathOperator{\RLmin}{RLmin}

\DeclareMathOperator{\sg}{sg}
\DeclareMathOperator{\height}{height}

\DeclareMathOperator{\lac}{{}^{\underbracket{\hspace{1em}}}}

\author{Alexander Burstein}

\title[Descent tops and bottoms in restricted permutations]{Distribution of sets of descent tops and descent bottoms on restricted permutations}

\affiliation{
Howard University, Washington, DC, USA
}


\begin{document}

\publicationdata{vol. 26:1, Permutation Patterns 2023}{2025}{8}{10.46298/dmtcs.12636}{2023-12-02; 2023-12-02; 2024-12-25}{2025-01-07}

\maketitle

\begin{abstract}
~\\
In this note, we prove some and conjecture other results regarding the distribution of descent top and descent bottom sets on some pattern-avoiding permutations. In particular, for 3-letter patterns, we show bijectively that the set of descent tops and the set of descent bottoms are jointly equidistributed on the avoiders of 231 and 312. We also conjecture similar equidistributions for 4-letter patterns, in particular, that the set of descent tops and the set of descent bottoms are jointly equidistributed on the avoiders of 3142, 3241, 4132. This conjecture and several others made in this paper have now been proved by Zhou, Zang, and Yan (2024).
\end{abstract}

\section{Preliminaries} \label{sec:prelim}

We begin with a few basic definitions related to permutations and permutation patterns, where we largely follow Bevan~\cite{Bevan15}. 

For integers $i$ and $j$, let $[i,j]=\{m\in\mathbb{Z}\mid i\le m\le j\}$, and let $[j]=[1,j]$. For a sequence $\sigma$, let $|\sigma|$ be the length (or size) of $\sigma$. We say that sequences $\sigma$ and $\tau$ are \emph{order-isomorphic} if $|\sigma|=|\tau|$ and, for all $i,j\in[1,|\sigma|]$, $\sigma(i)<\sigma(j)$ if and only if $\tau(i)<\tau(j)$. Given permutations $\sigma$ and $\pi$, we say that $\sigma$ \emph{contains} (an \emph{occurrence}, or an \emph{instance}) of \emph{pattern} $\pi$ if $\pi$ is order-isomorphic to a (not necessarily consecutive) subsequence of $\sigma$, denoted $\pi\preccurlyeq\sigma$. Otherwise, we say that $\sigma$ \emph{avoids} $\pi$. Let $S_n$ be the set of all permutations of length $n$, then we say that $\Av_n(\pi)=\{\sigma\in S_n\mid \pi\not\preccurlyeq\sigma\}$ is the set of \emph{$n$-avoiders} of $\pi$, and $\Av(\pi)=\cup_{n=0}^{\infty}\Av_n(\pi)$ is the set of \emph{avoiders} of $\pi$.

We say that $[i,j]$ is an \emph{interval} of a permutation $\sigma\in S_n$ if $i,j\in[n]$, $i\le j$, and $\{\sigma(m)\mid m\in[i,j]\}=[k,l]$ for some $k,l\in[n]$ (so that $l-k=j-i$). Given a permutation $\sigma\in S_n$ and nonempty permutations $\tau_1,\dots,\tau_n$, we say that the \emph{inflation} $\sigma[\tau_1,\dots,\tau_n]$ of $\sigma$ by $\tau_1,\dots,\tau_n$ is the permutation obtained by replacing each entry $\sigma(i)$, $i\in[n]$, with an interval order-isomorphic to $\tau_i$. For example, $231=21[12,1]$ and $312=21[1,12]$. We also call $\sigma\oplus\tau=12[\sigma,\tau]$ the \emph{direct sum} of $\sigma$ and $\tau$, and $\sigma\ominus\tau=21[\sigma,\tau]$ the \emph{skew-sum} of $\sigma$ and $\tau$ (see Figure~\ref{fig:sum-skewsum}). For example, $231=12\ominus 1=(1\oplus 1)\ominus 1$.

\begin{figure}[!ht]
\begin{center}
\begin{tikzpicture}[scale=0.3]
\begin{scope}[xshift=0, yshift=0, local bounding box = sum]

\draw (0,2) [left] node {$\sigma\oplus\tau=$};
\draw [very thick] ++(0,0) +(0,0) rectangle +(2,2);
\draw ++(0,0) +(1,1) node {$\sigma$};
\draw [very thick] ++(2,2) +(0,0) rectangle +(2,2);
\draw ++(2,2) +(1,1) node {$\tau$};

\end{scope}

\begin{scope}[xshift=15cm, yshift=0, local bounding box = skewsum]

\draw (0,2) [left] node {$\sigma\ominus\tau=$};
\draw [very thick] ++(0,2) +(0,0) rectangle +(2,2);
\draw ++(0,2) +(1,1) node {$\sigma$};
\draw [very thick] ++(2,0) +(0,0) rectangle +(2,2);
\draw ++(2,0) +(1,1) node {$\tau$};

\end{scope}
\end{tikzpicture}
\end{center}
\caption{Direct sum $\sigma\oplus\tau$ and skew-sum $\sigma\ominus\tau$ of permutations $\sigma$ and $\tau$.} \label{fig:sum-skewsum}
\end{figure}

Given a permutation $\sigma$, and a position $i$ such that $\sigma(i)>\sigma(i+1)$, we call $i$ a \emph{descent} of $\sigma$, $\sigma(i)$ a \emph{descent top} of $\sigma$, and $\sigma(i+1)$ a \emph{descent bottom} of $\sigma$. Likewise, if $\sigma(i)<\sigma(i+1)$, then we call $i$ an \emph{ascent} of $\sigma$, $\sigma(i)$ an \emph{ascent bottom} of $\sigma$, and $\sigma(i+1)$ an \emph{ascent top} of $\sigma$. Sometimes, when the context is clear, we may also refer to the ordered pair $\sigma(i)\sigma(i+1)$ as a descent (or ascent) and refer to $i$ as a descent (or ascent) position. We also call a maximal sequence of consecutive descents (respectively, ascents) a \emph{descent run} (respectively, \emph{ascent run}).

We say that patterns $\pi_1$ and $\pi_2$ are \emph{Wilf-equivalent}, denoted $\pi_1\sim\pi_2$, if for all $n\ge 0$, we have $|\Av_n(\pi_1)|=|\Av_n(\pi_2)|$. More recently, Sagan and Savage~\cite{SS12} defined a more refined version of Wilf-equivalence. For a permutation statistic $\st$, we say that $\pi_1$ and $\pi_2$ are \emph{$\st$-Wilf equivalent} if there is a bijection $\Theta:\Av_n(\pi_1)\to\Av_n(\pi_2)$ for all $n\ge 0$ that preserves the statistic $\st$, i.e.\ $\st=\st\circ\,\Theta$. We denote this by $\pi_1\stackrel{\st}{\leftrightsquigarrow}\pi_2$.

For a permutation $\sigma$, we define the following sets (which can also be thought of as permutation statistics):
\begin{itemize}
\item $\Des(\sigma)=\{i\mid\sigma(i)>\sigma(i+1)\}$, the \emph{descent set} of $\sigma$,
\item $\Destop(\sigma)=\{\sigma(i)\mid i\in\Des(\sigma)\}$, the \emph{descent top set} of $\sigma$,
\item $\Desbot(\sigma)=\{\sigma(i+1)\mid i\in\Des(\sigma)\}$, the \emph{descent bottom set} of $\sigma$.
\end{itemize}
Similar sets can be defined for the ascent statistic:
\begin{itemize}
\item $\Asc(\sigma)=\{i\mid\sigma(i)<\sigma(i+1)\}$, the \emph{ascent set} of $\sigma$,
\item $\Ascbot(\sigma)=\{\sigma(i)\mid i\in\Asc(\sigma)\}$, the \emph{ascent bottom set} of $\sigma$,
\item $\Asctop(\sigma)=\{\sigma(i+1)\mid i\in\Asc(\sigma)\}$, the \emph{ascent top set} of $\sigma$.
\end{itemize}
Note that $\Des(\sigma)\cup\Asc(\sigma)=[|\sigma|-1]$, $\Destop(\sigma)\cup\Ascbot(\sigma)=[|\sigma|]\setminus\{\sigma(|\sigma|)\}$, $\Desbot(\sigma)\cup\Asctop(\sigma)=[|\sigma|]\setminus\{\sigma(1)\}$.

Dokos et al.~\cite{DDJSS12} showed that $132\stackrel{\Des}{\leftrightsquigarrow}231$ and $213\stackrel{\Des}{\leftrightsquigarrow}312$. We will briefly outline the proof of the first statement, as the second one is very similar. It is well known that, for a nonempty permutation $\sigma$, we have
\[
\begin{split}
\sigma\in\Av(132) &\iff \sigma=231[\sigma',1,\sigma''] \text{ for some } \sigma',\sigma''\in\Av(132),\\
\sigma\in\Av(231) &\iff \sigma=132[\sigma',1,\sigma''] \text{ for some } \sigma',\sigma''\in\Av(231).
\end{split}
\]
Then the following map $\phi:\Av(132)\to\Av(231)$ is a $\Des$-preserving bijection. Let $\sigma\in\Av(132)$, then $\sigma=\emptyset$ or $\sigma=231[\sigma',1,\sigma'']$ for some $\sigma',\sigma''\in\Av(132)$. Then we let $\phi(\emptyset)=\emptyset$, and
\[
\phi(231[\sigma',1,\sigma''])=132[\phi(\sigma'),1,\phi(\sigma'')] \qquad \text{if } \sigma\ne\emptyset.
\]
Indeed, $\sigma'$ and $\phi(\sigma')$ occupy the same positions, as do $\sigma''$ and $\phi(\sigma'')$, so their descent positions are preserved under $\phi$. The only remaining descent exists if and only if $\sigma''\ne\emptyset$, and is the position of the singleton $1$ that corresponds to the top value $|\sigma|$ of both $\sigma$ and $\phi(\sigma)$, so that descent position is also preserved.

Dokos et al.~\cite{DDJSS12} also conjectured that $3142$, $3241$, $4132$ are $\maj$-Wilf equivalent, as are $1423$, $2314$, $2413$, where $\maj(\sigma)=\sum_{i\in\Des(\sigma)}{i}$ is the \emph{major index} of $\sigma$. This was later refined by Bloom~\cite{Bloom14}, who showed that, in fact, $1423$, $2314$, $2413$ are $\Des$-Wilf equivalent, as are $3142$, $3241$, $4132$.

In this paper, we will consider the distribution of $\Destop$ and $\Desbot$ on the same sets. We will show that $132$, $231$, $312$ are $\Destop$-Wilf equivalent, and furthermore, that $231$ and $312$ are $(\Destop,\Desbot)$-Wilf equivalent, then conjecture $\Destop$- and $(\Destop,\Desbot)$-Wilf equivalences on permutations of length $4$.

Note that $\Desbot$-Wilf equivalences follow from $\Destop$-Wilf equivalences by symmetries of the square. Define the following maps on $S_n$ for each $n\ge 0$:
\begin{itemize}

\item \emph{reversal}, $r:\sigma\mapsto\sigma^r$, where $\sigma^r(i)=\sigma(n+1-i)$,

\item \emph{complement}, $c:\sigma\mapsto\sigma^c$, where $\sigma^c(i)=n+1-\sigma(i)$,

\item \emph{inverse}, $\sigma\mapsto\sigma^{-1}$.

\end{itemize} 

The group of bijections generated by reversal, complement, and inverse are called symmetries of the square. In particular, the reversal of complement $r\circ c=c\circ r: \sigma\mapsto\sigma^{rc}$, given by $\sigma^{rc}(i)=n+1-\sigma(n+1-i)$ maps descents to descents and ascents to ascents. More precisely, 
\[
\begin{split}
i\in\Des(\sigma) &\iff n-i\in\Des(\sigma^{rc}),\\
i\in\Destop(\sigma) &\iff n+1-i\in\Desbot(\sigma^{rc}),\\
i\in\Desbot(\sigma) &\iff n+1-i\in\Destop(\sigma^{rc}),\\
\end{split}
\]
and therefore,
\[
\begin{split}
\sigma\stackrel{\Destop}{\leftrightsquigarrow}\tau &\iff \sigma^{rc}\stackrel{\Desbot}{\leftrightsquigarrow}\tau^{rc},\\
\sigma\stackrel{(\Destop,\,\Desbot)}{\leftrightsquigarrow}\tau &\iff \sigma^{rc}\stackrel{(\Destop,\,\Desbot)}{\leftrightsquigarrow}\tau^{rc}.
\end{split}
\]

\section{$\Destop$- and $\Desbot$-Wilf equivalence in $S_3$} \label{sec:res}

For patterns of length $3$, our results below can be summarized as follows: patterns that have the same descent top set (respectively, the same descent bottom set) are $\Destop$-Wilf equivalent (respectively, $\Desbot$-Wilf equivalent). Moreover, patterns that have both the same descent top set and the same descent bottom set are $(\Destop,\Desbot)$-Wilf equivalent.

\begin{theorem} \label{thm:destop-132-231-312}
Patterns $132$, $231$, and $312$ are $\Destop$-Wilf equivalent. Equivalently, patterns $213$, $231$, and $312$ are $\Desbot$-Wilf equivalent.
\end{theorem}

It is easy to see that these are the only $\Destop$-Wilf equivalences in $S_3$, since each of the remaining patterns in $S_3$ has a different set of descent tops, that is $\Destop(132)=\Destop(231)=\Destop(312)=\{3\}$, while $\Destop(123)=\emptyset$, $\Destop(213)=\{2\}$, $\Destop(321)=\{2,3\}$, so all other potential $\Destop$-Wilf equivalences fail at $n=3$.

\begin{proof}
The $\Destop$-preserving bijections in this proof are very similar to the $\Des$-preserving bijection $\phi$ in the previous section.

To show that $132\stackrel{\Destop}{\leftrightsquigarrow}231$, we use the same block decomposition as for $\Des$-Wilf equivalence. Let $\sigma\in\Av(132)$, then $\sigma=\emptyset$ or $\sigma=231[\sigma',1,\sigma'']$ for some $\sigma',\sigma''\in\Av(132)$. Then define a bijection $\Phi:\Av(132)\to\Av(231)$ recursively as follows. Let $\Phi(\emptyset)=\emptyset$ if $\sigma=\emptyset$, and
\begin{equation} \label{eq:Phi}
\begin{split}
\Phi(\sigma'\oplus 1)=\Phi(\sigma')\oplus 1  \qquad &\text{if } \sigma''=\emptyset,\\
\Phi(1\ominus\sigma'')=1\ominus\Phi(\sigma'') \qquad &\text{if } \sigma'=\emptyset,\\
\Phi(231[\sigma',1,\sigma''])=132[\Phi(\sigma''),1,\Phi(\sigma')] \qquad &\text{if } \sigma',\sigma''\ne\emptyset.
\end{split}
\end{equation}
Notice that this is different from the algorithm for $\phi$ in that there are three cases instead of one. This is because, in order for $\Phi$ to preserve descent tops, $\sigma'$ and $\sigma''$ only switch sides relative to the ``1'' (the top entry in $\sigma$) if both $\sigma'$ and $\sigma''$ are nonempty; otherwise, they must stay on the same side of the ``1''. Notice also that $\Phi$ moves the blocks corresponding to $\sigma'$ and $\sigma''$ horizontally, whereas $\phi$ (and $\Psi$ defined in \eqref{eq:Psi} below) moves those blocks vertically.

Since any statistic is clearly preserved if $\Phi(\sigma)=\sigma$, we may assume that $\sigma,\sigma',\sigma''\ne\emptyset$. Indeed, $\sigma'$ and $\Phi(\sigma')$ contain the same interval of values, as do $\sigma''$ and $\Phi(\sigma'')$, so their descent tops are preserved under $\Phi$. The only remaining descent top is the maximum value of $\sigma$ (and $\Phi(\sigma)$), which corresponds to the entry ``$3$'' in both $132$ and $231$, so that descent top is also preserved.

To show that $132\stackrel{\Destop}{\leftrightsquigarrow}312$, we use a slightly different block decomposition. 
It is well-known that, for a nonempty permutation $\sigma$, we have
\[
\begin{split}
\sigma\in\Av(132) &\iff \sigma=312[\sigma',\sigma'',1] \text{ for some } \sigma',\sigma''\in\Av(132),\\
\sigma\in\Av(312) &\iff \sigma=132[\sigma',\sigma'',1] \text{ for some } \sigma',\sigma''\in\Av(312).
\end{split}
\]
Let $\sigma\in\Av(132)$, then $\sigma=\emptyset$ or $\sigma=312[\sigma',\sigma'',1]$ for some $\sigma',\sigma''\in\Av(132)$. Define a bijection $\Psi:\Av(132)\to\Av(312)$ recursively as follows. Let $\Psi(\emptyset)=\emptyset$, and
\begin{equation} \label{eq:Psi}
\Psi(312[\sigma',\sigma'',1])=132[\Psi(\sigma''),\Psi(\sigma'),1].
\end{equation}
Again, assume that $\sigma\ne\emptyset$. From the definition of $\Psi$, we can see that $\Psi$ preserves the rightmost entry of the permutation $\sigma$. Furthermore, $\sigma'$ and $\Psi(\sigma')$ contain the same interval of values, as do $\sigma''$ and $\Psi(\sigma'')$, so their descent tops are preserved under $\Psi$. The only remaining descent top is the rightmost value of $\sigma'$, which is preserved by $\Psi$, and is thus also the rightmost value of $\Psi(\sigma')$. Note that there is a descent from the rightmost value of $\sigma'$ to the leftmost value of $\sigma''\oplus1$ if and only if $\sigma'\ne\emptyset$, and a descent from the rightmost value of $\Psi(\sigma')$ to the rightmost value of $\Psi(\sigma)$ (which is, in fact, $|\sigma''|+1$) also if and only if $\sigma'\ne\emptyset$. Thus, if $\sigma'\ne\emptyset$, then the rightmost value of $\sigma'$ (and of $\Psi(\sigma')$) is a descent top in both $\sigma$ and $\Psi(\sigma)$. This ends the proof.
\end{proof}

\begin{figure}[!ht]
\begin{center}
\begin{tikzpicture}[scale=0.5]

\begin{scope}[xshift=0, yshift=0, local bounding box = Phi]

\tikzstyle{arrow} = [very thick,|->]

\draw [help lines] (0,0) grid (5,5);

\draw [very thick] (0,0) rectangle (2,2);
\draw [very thick] (3,2) rectangle (5,4);
\filldraw (2.5,4.5) circle (4pt);
\draw (1,1) node {$\sigma_1$};
\draw (4,3) node {$\sigma_2$};

\draw [arrow] (6,2.5) -- (7,2.5) node [above left] {\large $\Phi$};

\draw ++(8,0) [help lines] +(0,0) grid +(5,5);

\draw [very thick] ++(8,0) +(3,0) rectangle +(5,2);
\draw [very thick] ++(8,0) +(0,2) rectangle +(2,4);
\filldraw ++(8,0) +(2.5,4.5) circle (4pt);
\draw ++(8,0) +(1,3) node {$\Phi(\sigma_2)$};
\draw ++(8,0) +(4,1) node {$\Phi(\sigma_1)$};

\draw (2.5,6) node (231) {$\Av(231)$};
\draw [arrow] (4.5,6) -- (8.5,6);
\draw ++(8,0) +(2.5,6) node (132) {$\Av(132)$};

\end{scope}

\begin{scope}[xshift=16cm, yshift=0, local bounding box = Psi]

\tikzstyle{arrow} = [very thick,|->]

\draw [help lines] (0,0) grid (5,5);

\draw [very thick] (0,3) rectangle (2,5);
\draw [very thick] (2,0) rectangle (4,2);
\filldraw (4.5,2.5) circle (4pt);
\draw (1,4) node {$\sigma_1$};
\draw (3,1) node {$\sigma_2$};

\draw [arrow] (6,2.5) -- (7,2.5) node [above left] {\large $\Psi$};

\draw ++(8,0) [help lines] +(0,0) grid +(5,5);

\draw [very thick] ++(8,0) +(0,0) rectangle +(2,2);
\draw [very thick] ++(8,0) +(2,3) rectangle +(4,5);
\filldraw ++(8,0) +(4.5,2.5) circle (4pt);
\draw ++(8,0) +(1,1) node {$\Psi(\sigma_2)$};
\draw ++(8,0) +(3,4) node {$\Psi(\sigma_1)$};

\draw (2.5,6) node (132) {$\Av(132)$};
\draw [arrow] (4.5,6) -- (8.5,6);
\draw ++(8,0) +(2.5,6) node (312) {$\Av(312)$};

\end{scope}
\end{tikzpicture}
\end{center}
\caption{Bijections $\Phi$ and $\Psi$ of Theorem~\ref{thm:destop-132-231-312}. Only the case $\sigma',\sigma''\ne\emptyset$ is shown for $\Phi$. For the remaining cases, see Figure~\ref{fig:sum-skewsum}.} \label{fig:destop-132-231-312}
\end{figure}

As we mentioned before, more can be asserted for patterns $231$ and $312$. Before we proceed, we will need another definition.

\begin{definition} \label{def:des-match}
Let $T=\{t_1<\dots<t_k\}$ and $B=\{b_1<\dots<b_k\}$ be (possibly empty) sets of positive integers with $|T|=|B|=k\ge 0$. We call the pair $(T,B)$ a \emph{descent matching} of size $k$ if $k=0$ and $T=B=\emptyset$ or $k>0$ and $t_i>b_i$ for all $i\in[k]$.
\end{definition}

\begin{theorem} \label{thm:231-312-destop-desbot}
Patterns $231$ and $312$ are $(\Destop,\Desbot)$-Wilf equivalent. Moreover, if $(T,B)$ is any descent matching of size $k$, and $n\ge\max(T)$ if $k>0$, or $n\ge 0$ if $k=0$, then there is a unique pair of permutations $\sigma\in\Av_n(312)$ and $\pi\in\Av_n(231)$ such that $\Destop(\sigma)=\Destop(\pi)=T$ and $\Desbot(\sigma)=\Desbot(\pi)=B$.
\end{theorem}

Note that the bijection we will define in the proof (which begins on page \pageref{proof:231-312-destop-desbot}) is different from $\Phi\circ\Psi^{-1}$ (or $\Psi\circ\Phi^{-1}$), which only preserves $\Destop$ but not $\Desbot$.

In order to prove Theorem~\ref{thm:231-312-destop-desbot}, it is helpful to restate it in terms of sets $\overline{T}=[n]\setminus T$ and $\overline{B}=[n]\setminus B$. Note that $i\in\overline{T}$ if and only if $i$ is an ascent bottom or is the rightmost entry in the permutation, and $i\in\overline{B}$ if and only if $i$ is an ascent top or is the leftmost entry in the permutation. Let us extend the definition of descent runs to include descent runs of length $1$, i.e. single entries that are neither descent tops nor descent bottoms. Then $\overline{B}$ is the set of \emph{descent run tops}, and $\overline{T}$ is the set of \emph{descent run bottoms}, or, in other words, the sets of initial and final entries, respectively, of descent runs (including those of length $1$).

\begin{lemma} \label{lem:des-match}
An ordered pair of sets $(T,B)$ is a descent matching if and only if for any $n\ge\max T$ there exists a permutation $\sigma\in S_n$ such that $T=\Destop(\sigma)$ and $B=\Desbot(\sigma)$.
\end{lemma}

\begin{proof}
If $T=B=\emptyset$, then let $\sigma=\emptyset\in S_0$. Now assume $|T|=|B|=k>0$ and let $n\ge\max T$. Suppose that $T=\{t_1<\dots<t_k\}$ and $B=\{b_1<\dots<b_k\}$. Form non-singleton descent runs from $T$ and $B$ so that, for each $i\in[k]$, the entry $t_i$ is followed by $b_i$. For each value $v$ in $\overline{T}\cap\overline{B}$, adjoin $v$ as a singleton descent run. Now concatenate all resulting descent runs in the increasing order of descent run tops (alternatively, in the increasing order of descent run bottoms). It is easy to check that this yields a permutation $\sigma\in S_n$ with $\Destop(\sigma)=T$ and $\Desbot(\sigma)=B$.

Conversely, let $\sigma\in S_n$ be a permutation with $\Destop(\sigma)=T=\{t_1<\dots<t_k\}$ and $\Desbot(\sigma)=B=\{b_1<\dots<b_k\}$. Let $i\in[k]$, and suppose that $t_i\le b_i$. Then each of the values $t_1,\dots,t_i$ must be followed by one of the values among $b_1,\dots,b_{i-1}$, which is impossible by the pigeonhole principle. Therefore, $t_i>b_i$ for each $i\in[k]$, i.e. $(T,B)$ is a descent matching.
\end{proof}

We also need to define left-to-right/right-to-left maxima/minima of a permutation. We say that an entry $\sigma(i)$ of a permutation $\sigma\in S_n$ is
\begin{itemize}
\item a \emph{left-to-right maximum} if $i=1$ or $\sigma(j)<\sigma(i)$ for all $j<i$.
\item a \emph{left-to-right minimum} if $i=1$ or $\sigma(j)>\sigma(i)$ for all $j<i$.
\item a \emph{right-to-left maximum} if $i=n$ or $\sigma(j)<\sigma(i)$ for all $j>i$.
\item a \emph{right-to-left minimum} if $i=n$ or $\sigma(j)>\sigma(i)$ for all $j>i$.
\end{itemize}
Let $\LRmax(\sigma)$ be the set of left-to-right maxima of $\sigma$, and define $\LRmin(\sigma)$, $\RLmax(\sigma)$, $\RLmin(\sigma)$ similarly.

\begin{lemma} \label{lem:destop-312}
For any permutation $\sigma\in\Av_n(312)$, $\overline{\Desbot(\sigma)}=\LRmax(\sigma)$. Equivalently, for any permutation $\sigma\in\Av_n(231)$, $\overline{\Destop(\sigma)}=\RLmin(\sigma)$.
\end{lemma}

\begin{proof}
If an element $\sigma(i)\in\Desbot(\sigma)$, then $\sigma(i-1)>\sigma(i)$, so $\sigma\notin\LRmax(\sigma)$. Thus, $\LRmax(\sigma)\subseteq\overline{\Desbot(\sigma)}$. 

Conversely, let $\sigma\in\Av_n(312)$ and suppose that elements of $\overline{\Desbot(\sigma)}$ occupy positions $i_1<\dots<i_\ell$ in $\sigma$. We claim that $\sigma(i_1)<\dots<\sigma(i_\ell)$. Indeed, assume that $\sigma(i_j)>\sigma(i_{j+1})$ for some $j\in[1,\ell-1]$. Since $\sigma(i_{j+1})$ is not a descent bottom, it is an ascent top, i.e.\ $\sigma(i_{j+1}-1)<\sigma(i_{j+1})$. Therefore, $\sigma(i_j)\sigma(i_{j+1}-1)\sigma(i_{j+1})$ is an instance of pattern $312$. But this is impossible, since $\sigma\in\Av_n(312)$. Therefore, $\sigma(i_1)<\dots<\sigma(i_\ell)$.

Now suppose that some $\sigma(i_j)\notin\LRmax(\sigma)$ for some $j\in[1,\ell]$, that is $\sigma(i_j)<\sigma(i')$ for some $i'<i_j$. Then $\sigma(i')$ belongs to some descent run of $\sigma$ with descent run top $\sigma(i'')$ for some $i''\le i'$. But then $\sigma(i'')\in\overline{\Desbot(\sigma)}$, $i''<i_j$, and $\sigma(i'')>\sigma(i_j)$, which is impossible by the argument in the previous paragraph. Therefore, $\overline{\Desbot(\sigma)}\subseteq\LRmax(\sigma)$, and thus, $\overline{\Desbot(\sigma)}=\LRmax(\sigma)$

Finally, we can obtain the second equality from the first by using the reverse complement operation. In other words, if $\sigma\in\Av_n(231)$, then $\sigma^{rc}\in\Av_n(312)$, so $\overline{\Desbot(\sigma)}=\LRmax(\sigma)$ and thus, $
\overline{\Destop(\sigma)}=\overline{\Desbot(\sigma^{rc})}=\LRmax(\sigma^{rc})=\RLmin(\sigma)$.
\end{proof}

\begin{proof}[of Theorem~\ref{thm:231-312-destop-desbot}] \label{proof:231-312-destop-desbot}
Let $(T,B)$ be a descent matching of size $k\ge 0$, and let $n\ge\max(T)$ if $k>0$, or $n\ge 0$ if $k=0$. As before, let $\overline{T}=[n]\setminus T$, $\overline{B}=[n]\setminus B$. Moreover, for each $i=1,\dots,n$, given a left prefix $\sigma(1)\dots\sigma(i-1)$ of a permutation $\sigma$, let $\Sigma_i=\{\sigma(j)\mid 1\le j<i\}$ (so $\Sigma_1=\emptyset$), and let $T_i=T\setminus\Sigma_i$, $B_i=B\setminus\Sigma_i$, $\overline{T}_i=\overline{T}\setminus\Sigma_i$, $\overline{B}_i=\overline{B}\setminus\Sigma_i$. Note that $T_i\cup\overline{T}_i=B_i\cup\overline{B}_i=[n]\setminus\Sigma_i=\{\sigma(j)\mid i\le j\le n\}$.

We claim that the following algorithm produces a permutation $\sigma\in\Av_n(312)$.
\begin{itemize}

\item Let $\sigma(1)=\min\overline{B}$.

\item For each $i$ from $1$ to $n-1$:

\begin{itemize}

\item If $\sigma(i)\in\overline{T}_{i}$, let $\sigma(i+1)=\min\overline{B}_{i+1}$.

\item If $\sigma(i)\in T_{i}$, let $\sigma(i+1)=\max\,\{m\in B_{i+1}\mid m<\sigma(i)\}$.

\end{itemize}

\end{itemize}

See Example~\ref{ex:destop-desbot} on page~\pageref{ex:destop-desbot} for an example of the application of the above algorithm.

We claim that every $\sigma(i)$ for $i\in[n]$ is well-defined if $(T,B)$ is a descent matching. Indeed, we have $\max B<\max T\le n$, so $n\in\overline{B}$, and thus $\overline{B}\ne\emptyset$ and $\sigma(1)$ is well-defined. In particular, this claim is true for $n=1$. Let $n\ge 2$, and suppose this claim is true for $n-1$. Let us construct $\sigma\in\Av_n(312)$ that corresponds to the pair $(T,B)$. We know that $\sigma(1)\in\overline{B}$, so $\sigma(1)\notin B$. 

\emph{Case 1.} Suppose $\sigma(1)\notin T_1=T$ as well. Define sets $T'$ and $B'$ as follows:  
\[
\begin{split}
T'&=\{t\mid t\in T \text{ and } t<\sigma(1)\}\cup\{t-1\mid t\in T \text{ and } t>\sigma(1)\}\\
B'&=\{b\mid b\in B \text{ and } b<\sigma(1)\}\cup\{b-1\mid b\in B \text{ and } b>\sigma(1)\},
\end{split}
\]
so $\sigma$ is well-defined as well. Note that $|T'|=|T|=|B|=|B'|$, so $(T',B')$ is also a descent matching. Let $\sigma'\in S_{n-1}$ be the permutation of length $n-1$ constructed from the pair $(T',B')$. By the induction hypothesis, we know that $\sigma'\in\Av_{n-1}(312)$, so, in particular, all the values in $[1,\sigma'(1)]$ occur in decreasing order from left to right in $\sigma'$. Then
\[
\sigma(i)=
\begin{cases}
\sigma'(i-1), & \text{if } i\ge 2 \text{ and } \sigma'(i-1)<\sigma(1),\\
\sigma'(i-1)+1, & \text{if } i\ge 2 \text{ and } \sigma'(i-1)\ge\sigma(1).
\end{cases}
\]
Thus, $\sigma\in\Av_n(312)$, since $\sigma'\in\Av_{n-1}(312)$ and no instance of pattern $312$ can start with $\sigma(1)$.

\emph{Case 2.} Now suppose that $\sigma(1)\in T_1=T$ (in particular, this implies that $\sigma(1)\ge 2$). Since $\sigma(1)=\min\overline{B}$, it follows that $[1,\sigma(1)-1]\subseteq B$, so $\sigma(2)=\sigma(1)-1$. With $T=\{t_1<\dots<t_k\}$ and $B=\{b_1<\dots<b_k\}$ as before, suppose $\sigma(1)=t_r$. Then $\sigma(2)=\sigma(1)-1=b_s$ for some $s\ge r$ (otherwise, if $s<r$, we would have $\sigma(1)=t_r>b_r>b_s=\sigma(1)-1$, which is impossible). Define $T''=\{t_1'',\dots,t''_{k-1}\}$ and $B''=\{b_1'',\dots,b''_{k-1}\}$ as follows:
\[
t_i''=
\begin{cases}
t_i, & \text{if } i<r,\\
t_{i+1}-1, & \text{if } i\ge r,
\end{cases}
\qquad \text{and} \qquad
b_i''=
\begin{cases}
b_i, & \text{if } i<s,\\
b_{i+1}-1, & \text{if } i\ge s.
\end{cases}
\]
To see that $t_1''<\dots<t''_{k-1}$ and $b_1''<\dots<b''_{k-1}$, we only need to check that $t_r''>t_{r-1}''$ and $b_s''>b_{s-1}''$. Indeed, $t_r''=t_{r+1}-1\ge t_r>t_{r-1}=t_{r-1}''$ and $b_s''=b_{s+1}-1\ge b_s>b_{s-1}=b_{s-1}''$. It is also easy to check that $t_i''>b_i''$ for $i<r$ and $t\ge s$. Moreover, for $r\le i<s$, we have $t_i''=t_{i+1}-1>t_r-1=\sigma(1)-1=b_s>b_i=b_i''$, so $t_i''>b_i''$ for $r\le i<s$ as well. Thus, $(T'',B'')$ is a descent matching. Let $\sigma''\in S_{n-1}$ be the permutation of length $n-1$ constructed from descent matching $(T'',B'')$. Then $\sigma(2)=\sigma(1)-1=\min\overline{B''}$ since $\sigma(1)-2\in B$ if $\sigma(1)-1>1$, and therefore, $\sigma''(1)=\sigma(2)=\sigma(1)-1$. By the induction hypothesis, we know that $\sigma''\in\Av_{n-1}(312)$, so, in particular, all the values in $[1,\sigma''(1)]$ occur in decreasing order from left to right in $\sigma''$. Thus, as in Case 1, we have
\[
\sigma(i)=
\begin{cases}
\sigma''(i-1), & \text{if } i\ge 2 \text{ and } \sigma''(i-1)<\sigma(1),\\
\sigma''(i-1)+1, & \text{if } i\ge 2 \text{ and } \sigma''(i-1)\ge\sigma(1).
\end{cases}
\]
Thus, $\sigma\in\Av_n(312)$, since $\sigma''\in\Av_{n-1}(312)$ and no instance of pattern $312$ can start with $\sigma(1)$.

This finishes the induction step, and therefore our claim is true for any $n\ge 1$.

Note that $\sigma(i+1)\in B$ and $\sigma(i)\in T$ exactly when $\sigma(i)>\sigma(i+1)$, i.e.\ when $i\in\Des(\sigma)$. Thus, $\Destop(\sigma)=T$ and $\Desbot(\sigma)=B$, and moreover, $\overline{B}=\LRmax(\sigma)$.


Finally, given a permutation $\pi\in\Av(231)$, we have $\pi^{rc}\in\Av(312)$. Let $T=\Destop(\pi)$, $B=\Desbot(\pi)$, and define $T^c=\{n+1-i\mid i\in T\}$, $B^c=\{n+1-i\mid i\in B\}$, then $\Destop(\pi^{rc})=B^c$, $\Desbot(\pi^{rc})=T^c$, so $\pi^{rc}$ is uniquely determined by the argument above, and thus, so is $\pi$.
\end{proof}

In fact, given a descent matching $(T,B)$, it is straightforward to see that we can find the associated $\pi\in\Av(231)$ directly using the following algorithm. For each $i=1,\dots,n$, given a right suffix $\pi(i+1)\dots\pi(n)$ of a permutation $\pi$, let $\Pi_{(i)}=\{\pi(j)\mid i<j\le n\}$ be the set of its letters (so $\Pi_{(n)}=\emptyset$), and let $T_{(i)}=T\setminus\Pi_{(i)}$, $B_{(i)}=B\setminus\Pi_{(i)}$, $\overline{T}_{(i)}=\overline{T}\setminus\Pi_{(i)}$, $\overline{B}_{(i)}=\overline{B}\setminus\Pi_{(i)}$. Note that $T_{(i)}\cup\overline{T}_{(i)}=B_{(i)}\cup\overline{B}_{(i)}=[n]\setminus\Pi_{(i)}=\{\sigma(j)\mid 1\le j\le i\}$.
\begin{itemize}

\item Let $\sigma(n)=\max\overline{T}$.

\item For each $i$ from $n-1$ to $1$ (in decreasing order):

\begin{itemize}

\item If $\sigma(i+1)\in\overline{B}_{(i+1)}$, let $\sigma(i)=\max\overline{T}_{(i)}$.

\item If $\sigma(i+1)\in B_{(i+1)}$, let $\sigma(i)=\min\,\{m\in T_{(i)}\mid m>\sigma(i+1)\}$.

\end{itemize}

\end{itemize}

\begin{example} \label{ex:destop-desbot}
Let $T=\{2,5,8,9\}$, $B=\{1,2,3,7\}$, and $n=9$. Then $(T,B)$ is a descent matching (since $2>1$, $5>2$, $8>3$, $9>7$), so we obtain $\sigma=453687921\in\Av_9(312)$ and $\pi=921534687\in\Av_9(231)$ with $\Destop(\sigma)=\Destop(\pi)=T$ and $\Desbot(\sigma)=\Desbot(\pi)=B$. Note also that $\overline{B}=\{4,5,6,8,9\}=\LRmax(\sigma)$ and $\overline{T}=\{1,3,4,6,7\}=\RLmin(\pi)$. See Figure~\ref{fig:destop-desbot} for permutation diagrams of $\sigma$ and $\tau$.
\begin{figure}[!ht]
\begin{center}
\begin{tikzpicture}[scale=0.5]

\begin{scope}[xshift=0, yshift=0, local bounding box = 312]

\draw[help lines] +(0,0) grid +(8,8);


\foreach \x in {1,2,3,4,5,6,7,8,9} 
{
\node[below] at (\x-1,0) {\small $\x$};
}

\foreach \x in {1,2,3,4,5,6,7,8,9} 
{
\node[left] at (0,\x-1) {\small $\x$};
}

\foreach \x/\y in {1/4,2/5,3/3,4/6,5/8,6/7,7/9,8/2,9/1}
{
\filldraw (\x-1,\y-1) circle (4pt);
}

\draw [very thick] ++(-1,-1) 
+(2,5) -- +(3,3) 
+(5,8) -- +(6,7) 
+(7,9) -- +(8,2) -- +(9,1);

\node[below] at (4,-1) {$\sigma=453687921\in\Av_9(312)$};

\end{scope}

\begin{scope}[xshift=9.75cm, yshift=0, local bounding box = iff]

\node at (0,4) {$\longleftrightarrow$};

\end{scope}

\begin{scope}[xshift=12cm, yshift=0, local bounding box = 231]

\draw[help lines] +(0,0) grid +(8,8);


\foreach \x in {1,2,3,4,5,6,7,8,9} 
{
\node[below] at (\x-1,0) {\small $\x$};
}

\foreach \x in {1,2,3,4,5,6,7,8,9} 
{
\node[left] at (0,\x-1) {\small $\x$};
}

\foreach \x/\y in {1/9,2/2,3/1,4/5,5/3,6/4,7/6,8/8,9/7}
{
\filldraw (\x-1,\y-1) circle (4pt);
}

\draw [very thick] ++(-1,-1) 
+(1,9) -- +(2,2) -- +(3,1)
+(4,5) -- +(5,3) 
+(8,8) -- +(9,7) 
;

\node[below] at (4,-1) {$\tau=921534687\in\Av_9(231)$};

\end{scope}

\end{tikzpicture}
\end{center}
\caption{An instance of $(\Destop,\Desbot)$-preserving bijection from Theorem~\ref{thm:231-312-destop-desbot} between $\Av(312)$ and $\Av(231)$ with descent matching $(T,B)=(\{2,5,8,9\},\{1,2,3,7\})$ and size $n=9$.} \label{fig:destop-desbot}
\end{figure}
\end{example}

\section{Distribution of $(2\underline{31}, \underline{31}2)$} \label{sec:dist}

In this section, we will generalize the results of Section \ref{sec:res} to find the joint distribution of a pair of vincular (generalized) patterns on permutations with fixed sets of descent tops and descent bottoms. A pattern is called \emph{vincular} (or \emph{generalized}) if some of its entries in consecutive positions must also occupy consecutive positions in its occurrences. For example, in patterns denoted $2\underline{31}$ and $\underline{31}2$ (in older notation, $2\textrm{-}31$ and $31\textrm{-}2$, respectively), the entries corresponding to ``$3$'' and ``$1$'' must be in consecutive position in its occurrences. Moreover, let $2_i\underline{31}$ and $\underline{31}2_i$ be occurrences of patterns $2\underline{31}$ and $\underline{31}2$ where the ``$2$'' has value $i$. Finally, given a pattern $\tau$ and a permutation $\sigma$, let $(\tau)\sigma$ be the number of occurrences of $\tau$ in $\sigma$. This makes $(\tau)$ a permutation statistic.  Then,
\[
(2\underline{31})=\sum_{i\ge 1} (2_i\underline{31}), \qquad (\underline{31}2)=\sum_{i\ge 1} (\underline{31}2_i).
\]
Clearly, $(2_i\underline{31})=0$ and $(\underline{31}2_i)=0$ when $i=1$ and $i=|\tau|$, but we include these values so as to treat all  $i\ge 1$ uniformly.

For a set $X\subseteq\mathbb{Z}$, define 
\[
X_{<i}=\{j\in X\mid j<i\}, \qquad X_{\le i}=\{j\in X\mid j\le i\}.
\]
Given a descent matching $(T,B)$ and $n\ge|T|=|B|$, 
define the \emph{signature function} $\sg_{n,(T,B)}$ of a permutation $\sigma\in S_n$ with $(\Destop,\Desbot)\sigma=(T,B)$ as follows:
\begin{equation} \label{eq:sg}
\sg_{n,(T,B)}(i)=|B_{<i}|-|T_{\le i}|+1, \qquad 1\le i\le n.
\end{equation}
This generalizes the signature function defined for Dumont permutations of the first kind in Burstein et al.~\cite{BJVS}. Note that $\sg_{n,(T,B)}(i)=1$ for $i=1$ and $i\ge\max T$.

For an integer $i\ge 1$, define
\[
[i]_{p,q}=\frac{p^{i}-q^{i}}{p-q}=\sum_{\substack{i_1,i_2\ge 0\\i_1+i_2=i-1}}p^{i_1}q^{i_2}.
\]

\begin{lemma} \label{lem:signature}
For any descent matching $(T,B)$, any $n\ge 0$, any $i\in[n]$, and any $\sigma\in S_n$ such that $\Destop(\sigma)=T$ and $\Desbot(\sigma)=B$, we have
\begin{equation} \label{eq:signature}
(2_i\underline{31})\sigma+(\underline{31}2_i)\sigma=\sg_{n,(T,B)}(i)-1.
\end{equation}
\end{lemma}

\begin{proof}
The sum $(2_i\underline{31})\sigma+(\underline{31}2_i)\sigma$ counts subsequences $iba$ and $bai$ of $\sigma$ such that $a<i<b$ and $a$ immediately follows $b$. Consider the number of descents $ba$ such that $a<i<b$. The number of descent bottoms $a$ such that $a<i$ is $|B_{<i}|$. Of those descent bottoms, the ones with the corresponding descent tops $b\le i$ do not contribute to occurrences of either $(2_i\underline{31})\sigma$ or $(\underline{31}2_i)\sigma$. Thus, the number of descents $ba$ such that $a<i<b$ is $|B_{<i}|-|T_{\le i}|=\sg_{n,(T,B)}(i)-1$. Moreover, these descents partition the rest of $\sigma$ (in one-line notation) into $\sg_{n,(T,B)}(i)$ (possibly empty) blocks, and $i$ may occur in any one of these blocks. Finally, note that if $i$ occurs in block $c_i\in[\sg_{n,(T,B)}(i)]$ from the left, then 
\begin{equation} \label{eq:weight}
(2_i\underline{31})\sigma=\sg_{n,(T,B)}(i)-c_i, \qquad (\underline{31}2_i)\sigma=c_i-1. \qedhere
\end{equation}
\end{proof}

\begin{theorem} \label{thm:dist}
For any integer $n\ge 0$ and any descent matching $(T,B)$, we have the following:
\begin{equation} \label{eq:dist}
\sum_{\substack{\sigma\in S_n\\ \Destop(\sigma)=T\\ \Desbot(\sigma)=B}}{\prod_{i=1}^{n}\left(p_i^{(2_i\underline{31})\sigma}q_i^{(\underline{31}2_i)\sigma}\right)}=\prod_{i=1}^{n}[\sg_{n,(T,B)}(i)]_{p_i,q_i}.
\end{equation}
\end{theorem}

\begin{proof}
For $n=0$, the identity \eqref{eq:dist} just says that $1=1$, since the products are empty. Let $n\ge 1$ and let $h=(h_1,\dots,h_n)=\sg_{n,(T,B)}$. We will show that every permutation $\sigma$ is determined uniquely by the descent matching $(T,B)=(\Destop,\Desbot)\sigma$ and the $n$-tuple $c=(c_i)_{1\le i\le n}$ such that $c_i\in[h_i]$ for all $i\in[n]$. Then
\[
\sum_{\substack{\sigma\in S_n\\ \Destop(\sigma)=T\\ \Desbot(\sigma)=B}}{\prod_{i=1}^{n}\left(p_i^{(2_i\underline{31})\sigma}q_i^{(\underline{31}2_i)\sigma}\right)}=\prod_{i=1}^{n}\sum_{\substack{\sigma\in S_n\\ \Destop(\sigma)=T\\ \Desbot(\sigma)=B}}p_i^{(2_i\underline{31})\sigma}q_i^{(\underline{31}2_i)\sigma}=\prod_{i=1}^{n}[h_i]_{p_i,q_i}.
\]

To prove this, we use a slightly modified version of the Fran\c{c}on-Viennot~\cite{FV} bijection. Consider an arbitrary sequence $(c_1,\dots,c_n)$ such that $c_i\in[h_i]$ for each $i\in[n]$. Start with the empty permutation $\varepsilon$, represented by the string $\lac$, and iterate the following step for $i=1,2,\dots,n$: replace $c_i$-th leftmost occurrence of $\lac$ as follows:
\[
\lac \mapsto 
\begin{cases}
\lac i\lac, & \text{ if } i\in B\setminus T,\\
i\lac, & \text{ if } i\notin B\cup T,\\
\lac i, & \text{ if } i\in B\cap T,\\
i, & \text{ if } i\in T\setminus B.
\end{cases}
\]
Finally, delete the remaining $\lac$ to the right of the resulting string to obtain the unique permutation $\sigma\in S_n$ that satisfies \eqref{eq:weight} for all $i\in[n]$. 

Note that in Cases 3 and 4, when $i\in B\cap T$ or $i\in T\setminus B$, the descent for which $i$ is the descent top is not part of the $h_i-1$ descents contained in an instance of $2_i\underline{31}$ or $\underline{31}2_i$. Thus, the number of $\lac$ in those cases before the insertion of $i$ is $h_i+1$, so the rightmost $\lac$ (which is always at the end of the string) cannot be replaced in Cases 3 and 4 (this is the difference from the Fran\c{c}on-Viennot bijection~\cite{FV}). However, in Cases 1 and 2, when $i\in B\setminus T$ or $i\notin B\cup T$, the rightmost $\lac$ may be replaced, and the number of $\lac$ in those cases is $h_i$. In all the cases, the rightmost character of the string remains $\lac$, until it is deleted after step $n$. See Example~\ref{ex:dist} for an instance of this mapping.

Note also that, the number of occurrences $2_i\underline{31}$ or $\underline{31}2_i$ is as in Equation \eqref{eq:weight}, since every $\lac$ except the rightmost one corresponds to a descent in the resulting permutation $\sigma$. This also means that, at every step, there is a $\lac$ to the right of all inserted letters.

It is straightforward to see that this mapping is a bijection. Indeed, given a permutation $\sigma$, we can recover its descent matching $(T,B)$ and the $n$-tuple $c=(c_1,\dots,c_n)$ and the corresponding replacement cases as follows. Append $\lac$ to the right of $\sigma$, then for each $i$ from $n$ down to $1$, do the following. Let $s_i$ be the substring of $\sigma$ defined as follows:
\[
s_i =
\begin{cases}
\lac i\lac, & \text{ if $i$ is adjacent to $\lac$ on the left and on the right},\\
i\lac, & \text{ if $i$ is adjacent to $\lac$ only on the right},\\
\lac i, & \text{ if $i$ is adjacent to $\lac$ only on the left},\\
i, & \text{ if $i$ is not adjacent to $\lac$ on either side}.
\end{cases}
\]
Replace $s_i$ with $\lac$ and let $c_i=k$ if the new $\lac$ is the $k$-th leftmost in the resulting string. (In Example~\ref{ex:dist}, this amounts to proceeding from bottom to top of the second table.)

Finally, note that the insertion of each $i\in[n]$ contributes the factor of
\[
\sum_{c_i=1}^{h_i} p_i^{h_i-c_i}q_i^{c_i-1}=\frac{p_i^{h_i}-q_i^{h_i}}{p_i-q_i}=[h_i]_{p_i,q_i}
\]
to the righthand side of Equation~\eqref{eq:dist}.
\end{proof}

The above proof lets us make the following observation. As before, let $h_i=\sg_{n,(T,B)}(i)$ for $i=1,\dots,n$, $h=(h_1,\dots,h_n)$, and let $w=(w_1,\dots,w_n)$ be a word on the alphabet $\{u,l^*,l_*,d\}$ such that
\[
w_i=
\begin{cases}
u, & \text{ if } i\in B\setminus T,\\
l^*, & \text{ if } i\notin B\cup T,\\
l_*, & \text{ if } i\in B\cap T,\\
d, & \text{ if } i\in T\setminus B.
\end{cases}
\]
Then the triple $(w,h,c)$ is a \emph{restricted Laguerre history} (see, for example, Chen and Fu \cite{CF} for all relevant definitions). This corresponds to labeled \emph{2-Motzkin path} of length $n$ from $(0,0)$ to $(n,0)$ with unit steps $u=(1,1)$, $l^*=(1,0)$, $l_*=(1,0)$, $d=(1,-1)$. Let the \emph{height} of each edge $w_i$ be the second coordinate of its left endpoint. Then for each edge $w_i\in\{u,l^*,l_*,d\}$,
\[
h_i=
\begin{cases}
\height(w_i)+1, & \text{ if } w_i\in\{u,l^*\},\\
\height(w_i), & \text{ if } w_i\in\{l_*,d\},
\end{cases}
\]
and for each $i=1,\dots,n$, the edge $w_i$ is labeled with some $c_i\in[h_i]$.

Theorem~\ref{thm:dist} implies, in particular, that statistics $(2\underline{31}, \underline{31}2)$ and $(\underline{31}2, 2\underline{31})$ are equidistributed on the set $\{\sigma\in S_n\mid \Destop(\sigma)=T,\ \Desbot(\sigma)=B\}$ for any descent matching $(T,B)$. Moreover, avoiding the vincular pattern $2\underline{31}$ (respectively, $\underline{31}2$) is equivalent to avoiding the classical pattern $231$ (respectively, $312$) as in Section \ref{sec:res}, which implies Theorem~\ref{thm:231-312-destop-desbot} as a corollary. Theorem~\ref{thm:dist} also generalizes \cite[Theorem 2.4]{BJVS}, which is a special case of Theorem~\ref{thm:dist} for $n=2m$, $T=\{2,4,\dots,2m\}$, and $p_i=q_i=1$ for all $i\in[n]$.

Letting $p_i=p$, $q_i=q$ for all $i\in[n]$, we obtain the following immediate corollary.
\begin{corollary} \label{cor:dist}
For any integer $n\ge 0$ and any descent matching $(T,B)$, we have the following:
\begin{equation} \label{eq:dist-same}
\sum_{\substack{\sigma\in S_n\\ \Destop(\sigma)=T\\ \Desbot(\sigma)=B}}{p^{(2\underline{31})\sigma}q^{(\underline{31}2)\sigma}}=\prod_{i=1}^{n}[\sg_{n,(T,B)}(i)]_{p,q}.
\end{equation}
\end{corollary}

\begin{example} \label{ex:dist}
Let $T=\{2,5,8,9\}$, $B=\{1,2,3,7\}$, and $n=9$, as in Example~\ref{ex:destop-desbot}. Then we have
\begin{center}
\begin{tabular}{c||c|c|c|c|c|c|c|c|c}
  $i$                & 1 & 2 & 3 & 4 & 5 & 6 & 7 & 8 & 9 \\ \hline\hline
  $|B_{<i}|$         & 0 & 1 & 2 & 3 & 3 & 3 & 3 & 4 & 4 \\ \hline
  $|T_{\le i}|$      & 0 & 1 & 1 & 1 & 2 & 2 & 2 & 3 & 4\\ \hline
  $\sg_{n,(T,B)}(i)$   & 1 & 1 & 2 & 3 & 2 & 2 & 2 & 2 & 1
\end{tabular}
\end{center}
so
\[
\begin{split}
\sum_{\substack{\sigma\in S_9\\ \Destop(\sigma)=\{2,5,8,9\}\\ \Desbot(\sigma)=\{1,2,3,7\}}}\!\!\!\!\!\!\!\!\!\!\!\!\!\!\!\!{p^{(2\underline{31})\sigma}q^{(\underline{31}2)\sigma}}&=\prod_{i=1}^{9}[\sg_{9,(\{2,5,8,9\},\{1,2,3,7\})}(i)]_{p,q}\\
&=[1]_{p,q}[1]_{p,q}[2]_{p,q}[3]_{p,q}[2]_{p,q}[2]_{p,q}[2]_{p,q}[2]_{p,q}[1]_{p,q}\\
&=1\cdot 1\cdot(p+q)(p^2+pq+q^2)(p+q)(p+q)(p+q)(p+q)\cdot 1\\
&=p^7 + 6 p^6 q + 16 p^5 q^2 + 25 p^4 q^3 + 25 p^3 q^4 + 16 p^2 q^5 + 6 p q^6 + q^7.
\end{split}
\]
Note that the permutations $\sigma=453687921\in\Av_9(312)$ and $\pi=921534687\in\Av_9(231)$ in Example~\ref{ex:destop-desbot} correspond to the terms $p^7$ and $q^7$, respectively, or equivalently, to $c=(1,\dots,1)$ and $c=h=\sg_{n,(T,B)}$, respectively.

Finally, let $h=\sg_{9,(\{2,5,8,9\},\{1,2,3,7\})}=(1,1,2,3,2,2,2,1)$ and choose an $n$-tuple $c=(c_i)_{1\le i\le 9}=(1,1,1,2,1,2,1,2,1)\in\prod_{i=1}^{9}[h_i]$. Then the quadruple $(n,T,B,c)$, where $n=9$, $T=\{2,5,8,9\}$, $B=\{1,2,3,7\}$, corresponds to permutation $\sigma=534978216\in S_9$ obtained as follows (with the block containing $i$ that replaced the $c_{i-1}$-th $\lac$ in iteration $i$ marked red):
\begin{center}
\begin{tabular}{c||c|c|c|c||c|c|c|c|c}
$i$ & $|B_{<i}|$ & $|T_{\le i}|$ & $w_i$ & $h_i$ & $c_i$ & $\sigma\!\upharpoonright_{[i]}$ & $(2_i\underline{31})\sigma$ & $(\underline{31}2_i)\sigma$ & $h_i-1$  \\ \hline\hline
0 &  &  &  &  &   & $\lac$ &  & \\ \hline
1 & 0 & 0 & $u$ & 1 & 1 & $\mathcolor{red!90!black}{\lac 1\lac}$ & 0 & 0 & 0\\ \hline
2 & 1 & 1 & $l_*$ & 1 & 1 & $\mathcolor{red!90!black}{\lac 2}1\lac$ & 0 & 0 & 0\\ \hline
3 & 2 & 1 & $u$ & 2 & 1 & $\mathcolor{red!90!black}{\lac 3\lac} 21\lac$ & 1 & 0 & 1\\ \hline
4 & 3 & 1 & $l^*$ & 3 & 2 & $\lac 3\mathcolor{red!90!black}{4\lac} 21\lac$ & 1 & 1 & 2\\ \hline
5 & 3 & 2 & $d$ & 2 & 1 & $\mathcolor{red!90!black}{5}34\lac 21\lac$ & 1 & 0 & 1\\ \hline
6 & 3 & 2 & $l^*$ & 2 & 2 & $534\lac 21\mathcolor{red!90!black}{6\lac}$ & 0 & 1 & 1\\ \hline
7 & 3 & 2 & $u$ & 2 & 1 & $534\mathcolor{red!90!black}{\lac 7\lac} 216\lac$ & 1 & 0 & 1\\ \hline
8 & 4 & 3 & $d$ & 2 & 2 & $534\lac 7\mathcolor{red!90!black}{8}216\lac$ & 0 & 1 & 1\\ \hline
9 & 4 & 4 & $d$ & 1 & 1 & $534\mathcolor{red!90!black}{9}78216\lac$ & 0 & 0 & 0\\ 
\end{tabular}
\end{center}
Here $\sigma\!\!\upharpoonright_{[i]}$ denotes the subsequence of $\sigma$ on values in $[i]$, together with $\lac$ indicating where greater values will be inserted (except for the rightmost $\lac$, where greater values simply may be inserted).

Thus, the restricted Laguerre history
\[
\begin{tikzpicture}[scale=0.7]
\node[left] at (0,1) {$(w,h,\textcolor{red}{c})=$};
\draw[help lines] (0,0) grid (9,2);
\draw [very thick] (0,0) -- (1,1) -- (2,1) -- (3,2) -- (4,2) -- (5,1) -- (6,1) -- (7,2) -- (8,1) -- (9,0);
\node[above] at (0.25,0.25) {\small $u$};
\node[above] at (1.5,0.85) {\small $l_*$};
\node[above] at (2.25,1.25) {\small $u$};
\node[above] at (3.5,2) {\small $l^*$};
\node[above] at (4.75,1.25) {\small $d$};
\node[above] at (5.5,0.9) {\small $l^*$};
\node[above] at (6.25,1.25) {\small $u$};
\node[above] at (7.75,1.25) {\small $d$};
\node[above] at (8.75,0.25) {\small $d$};

\node[below] at (0.75,0.75) {\small \textcolor{red}{$1$}};
\node[below] at (1.5,1.1) {\small \textcolor{red}{$1$}};
\node[below] at (2.75,1.75) {\small \textcolor{red}{$1$}};
\node[below] at (3.5,2.1) {\small \textcolor{red}{$2$}};
\node[below] at (4.25,1.75) {\small \textcolor{red}{$1$}};
\node[below] at (5.5,1.1) {\small \textcolor{red}{$2$}};
\node[below] at (6.75,1.75) {\small \textcolor{red}{$1$}};
\node[below] at (7.25,1.75) {\small \textcolor{red}{$2$}};
\node[below] at (8.25,0.75) {\small \textcolor{red}{$1$}};
\end{tikzpicture}
\]
corresponds to the permutation $\sigma=534978216\in S_9$, whose permutation diagram is given below.
\[
\begin{tikzpicture}[scale=0.5]
\draw[help lines] +(0,0) grid +(8,8);


\foreach \x in {1,2,3,4,5,6,7,8,9} 
{
\node[below] at (\x-1,0) {\small $\x$};
}

\foreach \x in {1,2,3,4,5,6,7,8,9} 
{
\node[left] at (0,\x-1) {\small $\x$};
}

\foreach \x/\y in {1/5,2/3,3/4,4/9,5/7,6/8,7/2,8/1,9/6}
{
\filldraw (\x-1,\y-1) circle (4pt);
}

\draw [very thick] ++(-1,-1) 
+(1,5) -- +(2,3)
+(4,9) -- +(5,7) 
+(6,8) -- +(7,2) -- +(8,1) 
;
\end{tikzpicture}
\]
\end{example}

\section{$\Destop$-Wilf and $\Desbot$-Wilf equivalence in $S_4$} \label{sec:conj}

We have several conjectures regarding $\Destop$-Wilf and $\Desbot$-Wilf equivalence for patterns of length $4$. Their principal motivation comes from the following. A \emph{Dumont permutation of the first kind} is a permutation $\sigma$ of an even length $2n$ such that $\Destop(\sigma)=\{2i\mid i\in[n]\}$. Jones~\cite{Jones}, Burstein and Jones~\cite{BJ21}, and Archer and Lauderdale~\cite{AL19} together conjectured Wilf-equivalences of patterns of length 4 on Dumont permutations of the first kind. We generalize these conjectures by claiming that those are exactly the nontrivial $\Destop$-Wilf equivalences for patterns of length $4$ on all permutations.

\begin{conjecture} \label{conj:destop}
The non-singleton $\Destop$-Wilf equivalence classes in $S_4$ are:
\begin{itemize}
\item $\{1243,3412\}$,

\item $\{1423,2413\}$,

\item $\{2143,3421\}$,

\item $\{2314,3124\}$,

\item $\{2431,3142,3241,4132\}$.

\end{itemize}
\end{conjecture}

Taking the reverse complement of the above patterns yields an equivalent conjecture:
\begin{conjecture} \label{cor:desbot}
The non-singleton $\Desbot$-Wilf equivalence classes in $S_4$ are:
\begin{itemize}
\item $\{2134,3412\}$,

\item $\{2314,2413\}$,

\item $\{2143,4312\}$,

\item $\{1423,1342\}$,

\item $\{3142,3241,4132,4213\}$.

\end{itemize}
\end{conjecture}

Notice that Conjectures~\ref{conj:destop} and~\ref{cor:desbot} imply that $3142$, $3241$, $4132$ are both $\Destop$-Wilf and $\Desbot$-Wilf equivalent. For those patterns, we have an even stronger conjecture.

\begin{conjecture} \label{conj:destop-desbot}
Patterns $3142$, $3241$, $4132$ are $(\Destop,\Desbot)$-Wilf equivalent.
\end{conjecture}

Both Conjectures \ref{conj:destop} and \ref{conj:destop-desbot} have been verified for avoiders of length $n\le 10$ with the help of Michael Albert's \emph{PermLab} software, see Albert~\cite{Albert}.

Note that Conjectures~\ref{conj:destop},~\ref{cor:desbot}, and~\ref{conj:destop-desbot} together imply that $\{3142,3241,4132\}$ is the unique non-singleton $(\Destop,\Desbot)$-Wilf equivalence class in $S_4$. Conjecture~\ref{conj:destop-desbot} also parallels the result of Bloom~\cite{Bloom14} that patterns $3142$, $3241$, $4132$ are $\Des$-Wilf equivalent.


%
%
%
%

Two more cases of Conjecture~\ref{conj:destop} appear to be part of families of $\Destop$-Wilf equivalences, which we conjecture to be, in fact, \emph{shape}-Wilf equivalences, defined below following Stankova and West~\cite{SW}.

Let $\lambda=(\lambda_1,\lambda_2,\dots,\lambda_k)$, where $\lambda_1\ge \lambda_2\ge\dots\ge \lambda_k>0$, be a partition of an integer $n\ge 0$. A \emph{Ferrers board} is a bottom-left justified arrangement of unit squares with $\lambda_i$ squares in each row $i=1,\dots,k$ (with rows numbered from the bottom up). A \emph{traversal} $T$ of a Ferrers board $F$ is a $(0,1)$-filling of the cells of $F$ with exactly one $1$ in each row and column. A subset of $1$s of $T$ forms a \emph{submatrix} of $F$ if all rows and columns of $F$ containing these $1$s intersect \emph{inside} $F$. We say that some $1$s of $T$ form an occurrence of a pattern $\sigma$ if the submatrix they form is equal to the \emph{permutation matrix} $M(\sigma)$ of $\sigma$ (i.e. $M(\sigma)$ has $1$s in positions $(i,\sigma(i))$ for $i\in[|\sigma|]$, with rows numbered from the bottom up and columns numbered left to right). If a traversal $T$ of a Ferrers board $F$ does not contain any submatrix equal to $M(\sigma)$, then we say that $T$ \emph{avoids} $\sigma$. We also denote the set of all traversals of $F$ that avoid $\sigma$ by $\Av_F(\sigma)$.

\begin{definition}
We say that patterns $\sigma_1$ and $\sigma_2$ are \emph{shape-Wilf-equivalent} if $|\Av_F(\sigma_1)|=|\Av_F(\sigma_2)|$ for any Ferrers board $F$.
\end{definition}

Note that the only boards $F$ of height $k$ which contain a traversal (call those boards \emph{traversable}) are boards such $\lambda_1=k$ (in other words, row 1 and column 1 have the same length) and $\lambda_i\ge k+1-i$ for $i=1,\dots,k$ (in other words, $F$ contains the staircase shape $(k,k-1,\dots,2,1)$).

Backelin et al.~\cite{BWX} proved the shape-Wilf equivalence of $\operatorname{id}_\ell=12\dots \ell$ and $r(\operatorname{id}_\ell)=\ell\dots 21$ for any $\ell\ge 1$, while Stankova and West~\cite{SW} proved that $231$ and $312$ are shape-Wilf equivalent.

Now we can extend $\st$-Wilf equivalence to Ferrers boards for all statistics $\st$ among $\Destop$, $\Desbot$, $\Asctop$, and $\Ascbot$.

Given a traversable Ferrers board $F$, let $\lambda_1>\lambda_2>\dots>\lambda_l$ be its distinct part sizes, and let $k_i$ be the multiplicity of $\lambda_i$ for each $i=1,\dots,l$, so that the shape of $F$ is $\lambda=(\lambda_1^{k_1},\lambda_2^{k_2},\dots,\lambda_l^{k_l})$ and $\sum_{i=1}^{l}{k_i}=k=\lambda_1$. Let $T$ be a traversal of $F$, and let $\tau\in S_k$ be a permutation such that $T=\{(i,\tau(i))\mid i\in[k]\}$. (In what follows, we will identify $T$ and $\tau$ and say that $\tau$ is a traversal of $F$.) 

For each $j\in[l]$, let $i_j=k_1+\dots+k_j$ (so $i_l=k$), and insert a separator $\vert$ after each $\tau(i_j)$ for $j=1,\dots,l$ (so the rightmost separator is at the end of $\tau$). Call the resulting string $\bar\tau_F$ the \emph{$F$-separation} of $\tau$. Now define the following statistics on $\bar\tau_F$:
\begin{itemize}

\item $\tau(i)$ is a \emph{descent top} of $\bar\tau_F$ if $\tau(i)\tau(i+1)$ is a block of $\bar\tau_F$ (so $\tau(i)$ is not immediately followed by a separator) and $\tau(i)>\tau(i+1)$,

\item $\tau(i)$ is a \emph{descent bottom} of $\bar\tau_F$ if $\tau(i-1)\tau(i)$ is a block of $\bar\tau_F$ (so $\tau(i)$ is not immediately preceded by a separator) and $\tau(i-1)>\tau(i)$, or $\vert\tau(i)$ is a block of $\bar\tau_F$ (so $\tau(i)$ is immediately preceded by a separator).

\item $\tau(i)$ is a \emph{ascent top} of $\bar\tau_F$ if $\tau(i-1)\tau(i)$ is a block of $\bar\tau_F$ (so $\tau(i)$ is not immediately preceded by a separator) and $\tau(i-1)<\tau(i)$,

\item $\tau(i)$ is a \emph{ascent bottom} of $\bar\tau_F$ if $\tau(i)\tau(i+1)$ is a block of $\bar\tau_F$  (so $\tau(i)$ is not immediately followed by a separator) and $\tau(i)<\tau(i+1)$, or $\tau(i)\vert$ is a block of $\bar\tau_F$ (so $\tau(i)$ is immediately followed by a separator).

\end{itemize}

Now define $\Destop(\bar\tau_F)$, $\Desbot(\bar\tau_F)$, $\Asctop(\bar\tau_F)$, and $\Ascbot(\bar\tau_F)$ to be the sets of descent tops, descent bottoms, ascent tops, and ascent bottoms of $\bar\tau_F$, respectively, and let $\Destop_F(\tau)=\Destop(\bar\tau_F)$ be the set of \emph{$F$-descent tops} of $\tau$, and similarly for the other statistics above.

\begin{example} \label{ex:f-sep}
Let $F$ be the Ferrers board for the partition $(6,6,5,5,3,3)$. Then $(i_1,i_2,i_3)=(2,4,6)$, so a traversal $\tau=465231$ of $F$ yields $\bar\tau_F=46\vert 52\vert 31\vert$, and thus
\[
\begin{split}
\Destop_F(\tau)&=\{3,5\},\\
\Desbot_F(\tau)&=\{1,2,3,5\},\\ 
\Asctop_F(\tau)&=\{6\},\\
\Ascbot_F(\tau)&=\{1,2,4,6\}.
\end{split} 
\]
\end{example}

Now, for a statistic $\st\in\{\Destop,\Desbot,\Asctop,\Ascbot\}$, we say that $\sigma_1$ and $\sigma_2$ are \emph{$\st$-shape-Wilf equivalent} if, for each traversable board $F$, there exists a bijection $\Theta:\Av_F(\sigma_1)\to\Av_F(\sigma_2)$ that preserves the statistic $\st_F$, i.e.\ $\st_F=\st_F\circ\,\Theta$.

\bigskip

Note that line 4 of Conjecture~\ref{conj:destop} claims that $231\oplus 1$ and $312\oplus 1$ are $\Destop$-Wilf equivalent. We can generalize this claim as follows.

\begin{conjecture} \label{conj:destop-231-312-shape}
Patterns $231\oplus\sigma$ and $312\oplus\sigma$ are $\Destop$-shape-Wilf equivalent for any permutation $\sigma$.
\end{conjecture}

Moreover, it appears that patterns $231$ and $312$ themselves are $(\Destop,\Desbot)$-shape-Wilf equivalent.

%

Similarly, note that the first and third patterns in line 5 of Conjecture~\ref{conj:destop} are $132\ominus 1$ and $213\ominus 1$. Taking their complements (which turns descent tops into ascent bottoms), we obtain the claim that $312\oplus 1$ and $231\oplus 1$ are $\Ascbot$-shape-Wilf equivalent. We can generalize this claim as follows.

\begin{conjecture} \label{conj:ascbot-231-312-shape}
Patterns $231\oplus\sigma$ and $312\oplus\sigma$ are $\Ascbot$-shape-Wilf equivalent for any permutation $\sigma\ne\emptyset$.
\end{conjecture}

Conjectures~\ref{conj:destop-231-312-shape} and~\ref{conj:ascbot-231-312-shape} have been verified for $\sigma$ of size $k\le 2$ on avoiders of length $n\le 10$.

\subsection{Recent progress} \label{subsec:recent}

There has been recent progress in proving parts of Conjecture~\ref{conj:destop} as well as Conjecture~\ref{conj:destop-desbot}. Zhou et al.~\cite{ZZY24} proved the following results (related to one another by the reversal and complement maps).

\begin{theorem}(Zhou et al.~\cite[Theorems 1.4, 1.5, 1.6]{ZZY24}) \label{thm:zzy}
\begin{itemize}
\item The patterns $3142$ and $3241$ are $(\Destop,\Desbot,\Asctop,\Peak,\LRmax)$-Wilf-equivalent.
\item The patterns $3142$ and $4132$ are $(\Destop,\Desbot,\Ascbot,\Val,\RLmin)$-Wilf-equivalent.
\item The patterns $2413$ and $1423$ are $(\Destop,\Asctop,\Ascbot,\Peak,\RLmax)$-Wilf-equivalent.
\end{itemize}
\end{theorem}

Here $\Peak(\sigma)=\Destop(\sigma)\cap\Asctop(\sigma)$ is the set of \emph{peak} values of $\sigma$ , and $\Val(\sigma)=\Desbot(\sigma)\cap\Ascbot(\sigma)$ is the set of \emph{valley} values of $\sigma$ (also known as \emph{pinnacles} and \emph{vales} of $\sigma$, respectively). For completeness, we add the following easy corollary. 

\begin{corollary} \label{cor:zzy}
The patterns $2413$ and $2314$ are $(\Desbot,\Asctop,\Ascbot,\Val,\LRmin)$-Wilf-equivalent.
\end{corollary}

\begin{proof}
Permutations avoiding $2413$ and $2314$ are exactly the reversals of those avoiding $3142$ and $4132$. Applying the reversal yields the map
\[
(\Destop,\Desbot,\Ascbot,\Val,\RLmin) \mapsto (\Asctop,\Ascbot,\Desbot,\Val,\LRmin)
\]
on permutation statistics.
\end{proof}

\acknowledgments \label{sec:ack}
The author is grateful to the anonymous referees, whose detailed remarks and suggestions have greatly improved the presentation of the material.

\end{document}